\documentclass[12pt]{article}

\usepackage{amsfonts}
\usepackage{amsmath}
\usepackage{amscd}
\usepackage{amssymb}
\usepackage{amsthm}
\usepackage{amsxtra}
\usepackage{diagrams}

\newcommand{\bbC}    {\mathbb C}

\newcommand{\bbR}    {\mathbb R}
\newcommand{\bbF}     {\mathbb F}

\newcommand{\bbZ}    {\mathbb Z}
\newcommand{\bbN}    {\mathbb N}

\newcommand{\cA}    {{\cal A}}

\newcommand{\cK}    {{\cal K}}
\newcommand{\cL}    {{\cal L}}
\newcommand{\cM}    {{\cal M}}

\newcommand{\cS}    {{\cal S}}

\newcommand{\cU}    {{\cal U}}
\newcommand{\cV}    {{\cal V}}

\newcommand{\de}    {\delta}
\newcommand{\De}    {\Delta}

\newcommand{\io}    {\iota}
\newcommand{\ka}    {\kappa}

\newcommand{\om}    {\omega}

\newcommand{\Om}    {\Omega}

\newcommand{\vk}     {\varkappa}

\newcommand{\E}       {\mathrm E}

\newcommand{\I}         {\mathrm I}

\newcommand{\alg}       {\mathrm{alg}}

\newcommand{\Lie}       {\mathrm{Lie}}

\newcommand{\bE}     {\mathbf E}

\newcommand{\bG}      {\mathbf G}

\newcommand{\bX}    {\mathbf X}
\newcommand{\bY}    {\mathbf Y}

\newcommand{\bR}     {\mathbf R}

\newcommand{\fA}    {\mathfrak A}

\newcommand{\fD}    {\mathfrak D}

\newcommand{\fM}    {\mathfrak M}

\newcommand{\gl}     {\mathfrak{gl}}

\newcommand{\w}       {\wedge}
\newcommand{\ck}      {\check} 
\newcommand{\na}      {\nabla} 

\newcommand{\bDe}    {\pmb{\Delta}}

\newcommand{\bna}    {\pmb{\nabla}}

\DeclareMathOperator{\id}    {\,id}
\DeclareMathOperator{\Hom}   {Hom}
\DeclareMathOperator{\End}   {End}
\DeclareMathOperator{\Aut}   {Aut}

\DeclareMathOperator{\im}    {Im}
\DeclareMathOperator{\ke}    {Ker}
\DeclareMathOperator{\com}   {{\scriptstyle\circ}\,}
\DeclareMathOperator{\ad}    {ad}

\DeclareMathOperator{\ann}     {ann}

\newtheorem{theorem}{Theorem}

\newtheorem{prop}{Proposition}
\newtheorem{cor}{Corollary}
\theoremstyle{definition} 
\newtheorem{rem}{Remark} 

\newtheorem{defi}{Definition}

\begin{document}

\title{On analysis in algebras and modules}
\author
          {
             Zharinov V.V.
             \thanks{Steklov Mathematical Institute}
             \thanks{E-mail: zharinov@mi.ras.ru}
            }
\date{}
\maketitle

\begin{abstract}
An algebraic technique adapted to the problems of the fundamental theoretical 
physics is presented. The exposition is an elaboration and an extension of  the methods 
proposed in previous works by the author. 
\end{abstract} 

{\bf Keywords:} algebra, module, multiplicator, differentiation, 
covariant de\-rivative, gauge transform, moduli space, differential form, 
cohomology. 

\section{Introduction.} 
It is widely recognized that contemporary fundamental theoretical 
physics is based on geometrical and algebraic ideas, methods and 
constructions. Moreover, in our days any fundamental physical theory must be 
formulated purely in generalized geometrical setting, 
using a good amount of algebraic machinery. In this situation it is reasonable 
to work directly in algebra-geometrical language and to leave the actual 
physical model to the final study (calculations). To do this one needs an 
appropriate algebraic technique adapted to the problems of theoretical 
and mathematical physics. 

In this paper an approach to such adaptation is presented. 
The exposition is in fact an elaboration and an extension of constructions and methods 
developed in the papers \cite{Z1}, \cite{Z2} and lectures \cite{Z3}. 
More details, proofs and applications can be found there. 
We hope that our technique will be useful in different researches, 
such as  \cite{DT1}--\cite{BV}.

We use the following general notations: 
\begin{itemize} 
	\item 
		$\bbF=\bbR,\bbC$; 
	\item
		$\bbN=\{1,2,3,\dots\}\subset\bbZ_+=\{0,1,2,\dots\}
		\subset\bbZ=\{0,\pm1,\pm2,\dots\}$. 
\end{itemize} 

We, also, use the standard notation of homology theory (see, e.g., \cite{SM}).
In particular, 
\begin{itemize}
	\item 
		$\Hom(\cS;\cS')$ is the set of all mappings from a set $\cS$ to a set $\cS'$. 
	\item 
		$\Hom_\bbF(\cL;\cL')$ is the linear space of all linear mappings from 
		a linear space $\cL$ to a linear space $\cL'$; 			
	\item 
		$\Hom_\cA(\cM;\cM')$ is the $\cA$-module of all 
		$\cA$-linear mappings from an $\cA$-module $\cM$ to an 
		$\cA$-module $\cM'$, where 
		$\cA$ is an associative commutative algebra; 
	\item 
		$\Hom_{\alg}(\cA;\cA')$ is the set of all algebra morphisms  
		from $\cA$ to $\cA'$, where $\cA,\cA'$ are associative commutative 
		algebras; 	
	\item 
		$\Hom_{\Lie}(\fA;\fA')$ is the linear space of all Lie algebra 
		morphisms from a Lie algebra $\fA$ to a Lie algebra $\fA'$. 
\end{itemize} 
Note, 
\begin{itemize} 
	\item 
		the set $\End(\cS)=\Hom(\cS;\cS)$ has the structure of a monoid 
		under the composition rule of mappings; 
	\item 
		the set $\End_\bbF(\cL)=\Hom_\bbF(\cL;\cL)$ has the structure 
		of an unital asso\-ciative algebra under the composition rule of mappings; 
		the adjoint Lie algebra $\gl_\bbF(\cL)$ is defined by the 
		commutator rule of mappings; 
	\item 
		the set $\End_\cA(\cM)=\Hom_\cA(\cM;\cM)$ has the structure 
		of an unital asso\-ciative $\cA$-algebra under the composition rule of mappings, 
		the adjoint Lie $\cA$-algebra $\gl_\cA(\cM)$ is defined by the 
		commutator rule of mappings; 
	\item 
		the set $\End_{\alg}(\cA)=\Hom_{\alg}(\cA;\cA)$ has the structure 
		of a monoid under the composition rule of mappings; 
	\item 
		the set $\End_{\Lie}(\fA)=\Hom_{\Lie}(\fA;\fA)$ has the structure 
		of a monoid under the composition rule of mappings. 
\end{itemize}

All linear operations are done over the number field $\bbF$. 
The summation over repeated upper and lower indices is as a rule assumed. 
If objects under study have natural topologies, we assume that the corresponding 
mappings are continuous (for example, if $\cS$ and $\cS'$ are topological spaces, 
then $\Hom(\cS;\cS')$ is the set of all continuous mappings from $\cS$ to $\cS'$). 

\section{The multiplication in algebras and modules.}
\subsection{The multiplication in algebras.}
\begin{defi} 
Let $\cA$ be an associative commutative algebra. 
A linear mapping $R\in\End_\bbF(\cA)$ is called a {\it multiplicator 
of the algebra} $\cA$, if the {\it multiplicator rule} 
\begin{equation*}
R(f\cdot g)=f\cdot Rg \quad\text{holds for all}\quad f,g\in\cA, 
\end{equation*}
i.e., if $R\in\End_\cA(\cA)$. 
Let $\fM(\cA)=\End_\cA(\cA)$ be the set of all multiplicators of the algebra $\cA$. 
\end{defi} 

\begin{prop} 
Let $\cA$ be an associative commutative algebra. 
The following statements are valid: 
\begin{itemize} 
	\item 
		the set $\fM(\cA)$ is an unital subalgebra of the algebra $\End_\bbF(\cA)$; 
	\item 
		for any $R\in\fM(\cA)$ the kernel $\ke R=\{f\in\cA\mid R f=0\}$ 
		and the image $\im R=\{f=R g\mid g\in\cA\}$ are ideals of the algebra $\cA$ 
		(in other words, submodules of the $\cA$-module $\End_\cA(\cA)$); 
	\item 
		for any $R\in\fM(\cA)$ the image $R(\ann\cA)\subset\ann\cA$; 
	\item 
		the commutator $[R',R'']\in\Hom_\bbF(\cA;\ann\cA)$ 
		for all $R',R''\in\fM(\cA)$, 
		in par\-ti\-cular, the algebra $\fM(\cA)$ is commutative if $\ann\cA=0$ 
		(for example, if the algebra $\cA$ is unital). 
\end{itemize}
\end{prop} 
Remind, the {\it annihilator} 
$\ann\cA=\{f\in\cA\mid f\cdot\cA=0\}$ (i.e., $f\cdot g=0$ for all $g\in\cA$); 

\begin{prop}\label{ad} 
Let $\cA$ be an associative commutative algebra. 
The adjoint action defines the algebra morphism 
\begin{equation*} 
	\ad : \cA\to\fM(\cA), \quad f\mapsto\ad_f : \cA\to\cA, 
	\quad g\mapsto\ad_f g=f\cdot g.
\end{equation*} 
The kernel $\ke\ad=\{f\in\cA\mid \ad_f=0\}=\ann\cA$. 
In particular, the mapping $\ad : \cA\to\fM(\cA)$ is injection if $\ann\cA=0$.
\end{prop}

\begin{prop} 
Let $\cA$ be an unital associative commutative algebra, 
and let $e\in\cA$ be the unit element of $\cA$. 
The natural algebra isomorphism 
\begin{equation*} 
	e^*=\ad^{-1} : \fM(\cA)\simeq\cA, \quad R\mapsto e^*R=R e, 
\end{equation*} 
is defined.
\end{prop} 
In this case we may identify algebras $\fM(\cA)$ and $\cA$. 

\subsection{The multiplication in modules.}
\begin{defi}
Let $\cA$ be an associative commutative algebra, 
and let $\cM$ be an $\cA$-module. A pair 
$\bR=(\De_R,R)\in\E_\cA(\cM,\cA)=\End_\cA(\cM)\times\fM(\cA)$ 
is called a 
{\it multiplicator of the $\cA$-module $\cM$} if the {\it multiplicator rule}
\begin{equation*} 
	\De_R(f\cdot M)=Rf\cdot M 
	\quad\text{is valid for all}\quad f\in\cA, \ M\in\cM
\end{equation*} 
(remind, $\fM(\cA)=\End_\cA(\cA)$).
Let $\fM(\cM)=\End_{\cA\text{-}\cM}(\cM)$ be the set of all multiplicators 
(of all $\cA\text{-}\cM$-linear mappings) of the $\cA$-module $\cM$. 
\end{defi} 

Note, the direct product $\E_\cA(\cM,\cA)$ has the structure 
of an unital algebra with the component-wise multiplication, 
defined by the  composition rule, 
\begin{equation*} 
	\bR'\com\bR''=(\De_{R'}\com\De_{R''},R'\com R'') \quad\text{for all}\quad 
	\bR',\bR''\in\E(\cM,\cA), 
\end{equation*}
where $\bR'=(\De_{R'},R')$, $\bR''=(\De_{R''},R'')$, 
and the unit element $\mathbf{id}=(\id_\cM,\id_\cA)$. 

\begin{prop} 
Let $\cA$ be an associative commutative algebra, 
and let $\cM$ be an $\cA$-module. 
The following statements are valid: 
\begin{itemize} 
	\item 
		the set $\fM(\cM)$ is an unital subalgebra of the algebra $\E_\cA(\cM,\cA)$; 
	\item 
		for any $\bR=(\De_R,R)\in\fM(\cM)$ 
		the kernel $\ke\De_R=\{M\in\cM\mid \De_R M=0\}$ and 
		the image $\im\De_R=\{M=\De_R N\mid N\in\cM\}$ 
		are submodules of the $\cA$-module $\cM$; 
	\item 
		for any $\bR=(\De_R,R)\in\fM(\cM)$ 
		the image $\im\De_R\big(\!\ann_\cM\cA\big)\subset\ann_\cM\cA$, 
		and the image $\im\De_R\big(\!\ann_\cA\cM\big)\subset\ann_\cA\cM$; 
	\item 
		for any $\bR'=(\De_{R'},R'),\bR''(\De_{R''},R'')\in\fM(\cM)$ 
		the image $[\De_{R'},\De_{R''}](\cA\cdot\cM)=0$,  
		hence, 
		the commutator $[\De_{R'},\De_{R''}]\in\Hom_\cA(\cM;\ann_\cM\cA)$, 
		(the set $\cA\cdot\cM=\{f\cdot M\mid f\in\cA,M\in\cM\}\subset\cM$).
\end{itemize}
\end{prop}
Remind, the annihilator 
$\ann_\cM\cA=\{M\in\cM\mid \cA\cdot M=0\}$, while the annihilator 
$\ann_\cA\cM=\{f\in\cA\mid f\cdot\cM=0\}$. 

\begin{prop} 
Let $\cA$ be an associative commutative algebra, 
and let $\cM$ be an $\cA$-module. 
The adjoint action defines the algebra morphism 
\begin{equation*} 
	{\bf ad} : \cA\to\fM(\cM), \quad f\mapsto{\bf ad}_f=(\ad_f,\ad_f), 
\end{equation*}
where $\ad_f M=f\cdot M$, $\ad_f g=f\cdot g$, for all $g\in\cA$, $M\in\cM$. 
The kernel $\ke{\bf ad}=\ann_\cA\cM$.
\end{prop} 

\begin{prop} 
Let $\cA$ be an associative commutative algebra, and let $\cM$ be a free $\cA$-module 
with an $\cA$-basis ${\mathbf b}=\{b_i\in\cM\mid i\in I\}$, $I$ is an index set, 
so $\cM=\{M=M^i\cdot b_i\mid M^i\in\cA\}$. 
Then an algebra injection $\fM(\cA)\to\fM(\cM)$ is defined by the component-wise action, 
\begin{equation*} 
	\fM(\cA)\ni R\mapsto\bR=(\De_R,R)\in\fM(\cM), \quad \De_R(M^i\cdot b_i)=RM^i\cdot b_i.
\end{equation*}
\end{prop}

Let $\cA$ be an associative commutative algebra, 
and let $\cM$ be an $\cA$-module. 
Let us consider the mapping 
\begin{equation*}
	\pi : \fM(\cM)\to\fM(\cA), \quad \bR=(\De_R,R)\mapsto R. 
\end{equation*} 
{\it We assume that the mapping $\pi$ is the surjection}, i.e., 
$\im\pi=\fM(\cA)$. In this case, the inverse image 
$\pi^{-1}(R)=\{\bR=(\De_R,R)\in\fM(\cM)\}$ 
is defined for every $R\in\fM(\cA)$. 

We denote by 
\begin{equation*}
	\cS(\pi)=\big\{\bDe\in\Hom(\fM(\cA);\fM(\cM)) \ \big| \ 
	\pi\com\bDe=\id_{\fM(\cA)}\big\}
\end{equation*} 
the set of all {\it sections} of the the bundle $\pi$. 
In more detail, $R\mapsto\bDe_R=(\De_R,R)$ for all $R\in\fM(\cA)$. 
There are the following relevant subsections: 
\begin{equation*} 
	\cS_*(\pi)=\cS(\pi)\cap\Hom_*(\fM(\cA);\fM(\cM)), \quad *=\bbF, \ \cA. 
\end{equation*} 
The set $\cS_\bbF(\pi)$ has the structure of a linear space, 
the set $\cS_\cA(\pi)$ has the structure of an $\cA$-module. 
Elements of an $\cA$-module $\cS_\cA(\pi)$ may be called 
{\it covariant multiplicators}.  

Let a covariant multiplicator $\bDe\in\cS_\cA(\pi)$ be given. 
By construction, the composition $R'\com R''\in\fM(\cA)$ for any multiplicators $R',R''\in\fM(\cA)$.
Hence, the {\it residual} 
$\bG(\bDe)=(G(\bDe),0)\in\Hom_\cA(\otimes^2_\cA\fM(\cA);\fM(\cM))$ 
is defined, where
\begin{equation*} 
	G(\bDe)\in\Hom_\cA(\otimes^2_\cA\fM(\cA);\End_\cA(\cM),
	\quad G(\bDe)_{R',R''}=\De_{R'}\com\De_{R''}-\De_{R'\com R''}. 
\end{equation*} 

\begin{theorem}\label{T1} 
Let $\cA$ be an associative commutative algebra, 
and let $\cM$ be an $\cA$-module. 
For every $R\in\fM(\cA)$ the inverse image $\pi^{-1}(R)$ 
is an $\cA$-affine space over the $\cA$-module $\Hom_\cA(\cM;\ann_\cM\cA)$. 
Hence, the triple $\pi : \fM(\cM)\to\fM(\cA)$ is an $\cA$-affine fiber bundle over the 
$\cA$-module $\Hom_\cA(\cM;\ann_\cM\cA)$.
\end{theorem}
\begin{proof} 
Indeed, let $\bR'=(\De'_R,R),\bR''=(\De''_R,R)\in\pi^{-1}(R)$, 
then the difference $(\De'_R-\De''_R)(f\cdot M)=(R-R)f\cdot M=0=f\cdot(\De'_R-\De''_R)M$ 
for all $f\in\cA$ and $M\in\cM$, i.e., 
$\De'_R-\De''_R\in\Hom_\cA(\cM;\ann_\cM\cA)$. 
On the other side, let $\bR=(\De_R,R)\in\pi^{-1}(R)$ and 
$\rho\in\Hom_\cA(\cM;\ann_\cM\cA)$, then 
$(\De_R+\rho)(f\cdot M)=\De_R(f\cdot M)+\rho(f\cdot M)
=Rf\cdot M+f\cdot\rho M=Rf\cdot M+0=Rf\cdot M$ 
for all $f\in\cA$ and $M\in\cM$, i.e., $\bR'=(\De_R+\rho,R)\in\pi^{-1}(R)$.
\end{proof} 

Any associative commutative algebra $\cA$ has the adjoint structure 
of an $\cA$-module, where 
$\cA\times\cA\ni(f,M)\mapsto f\cdot M$. 
The resulting $\cA$-module we denote $\pmb\cA$. 

\begin{prop} 
Let $\cA$ be associative commutative algebra, 
and let $\pmb\cA$ be the adjoint $\cA$-module. Then 
\begin{itemize} 
	\item 
		there is the natural algebra injection \\
		$\io_\fM : \fM(\cA)\to\fM(\pmb\cA), \quad R\mapsto\bR=(R,R)$; 
	\item 
		$\fM(\pmb\cA)\!=\!\im\io_\fM\oplus\Hom_\cA(\pmb\cA;\ann\cA), 
		\ \bR\!=\!(\De_R,R)\!=\!(R,R)+(\De_R-R,0)$.
\end{itemize}
\end{prop}

\section{Hochschild cohomology.}
Let $\cA$ be an associative commutative algebra, let $\cK$, $\cM$ be $\cA$-modules, 
and let $\cU=\cA,\cK$, \ $\cV=\cA,\cM$. 

\begin{defi} 
The $\cA$-module $C(\cU,\cV)=\oplus_{q\in\bbZ}C^q(\cU,\cV)$ 
of {\it multiplicator co\-chains over the $\fM(\cU)$ with coefficients in 
the $\fM(\cV)$} is defined by the rule 
\begin{equation*} 
	C^q(\cU,\cV)=\begin{cases}
		0, & q<0, \\
		\fM(\cV), & q=0, \\
		\Hom_\cA(\otimes^q_\cA\fM(\cU);\fM(\cV)), & q>0.
					  \end{cases}
\end{equation*} 
\end{defi} 
In particular, the set $C(\cU,\cV)$ has the natural structure 
of a tensor $\cA$-algebra, where 
\begin{equation*} 
	(c^p\otimes c^q)(\eta_1,\dots,\eta_{p+q})
	=c^p(\eta_1,\dots,\eta_p)\com c^q(\eta_{p+1},\dots,\eta_{p+q})
\end{equation*}
for all $c^p\in C^p(\cU,\cV)$, $c^q\in C^q(\cU,\cV)$, 
$\eta_1,\dots,\eta_{p+q}\in\fD(\cU). $

\begin{defi}
Let a mapping $\ka\in C^1(\cU,\cV)=\Hom_\cA(\fM(\cU);\fM(\cV))$ be fixed. 
The endomorphism $\de=\de_\ka\in\End_\cA(C(\cU,\cV))$ is defined by the 
{\it Hochschild rule} 
\begin{align*} 
	\de c&(\eta_1,\dots,\eta_{q+1})=\ka\eta_1\com c(\eta_2,\dots,\eta_{q+1}) \\
	        &+\sum_{1\le r\le q}(-1)^r 
	          c(\eta_1,\dots\eta_r\com\eta_{r+1}\dots,\eta_{q+1})
	          +(-1)^{q+1}c(\eta_1,\dots,\eta_q)\com\ka\eta_{q+1}
\end{align*}
for all $q\in\bbZ_+$, $c\in C^q(\cU,\cV)$, $\eta_1,\dots,\eta_{q+1}\in\fM(\cU)$, 
in particular, the mapping 
$\de^q=\de\big|_{C^q(\cU,\cV)} : C^q(\cU,\cV)\to C^{q+1}(\cU,\cV)$. 
\end{defi}

\begin{theorem} 
Let a mapping $\ka\in C^1(\cU,\cV)$ be fixed. 
\begin{itemize} 
	\item 
		The endomorphism $\de$ is an exterior differentiation of the tensor 
		$\cA$-algebra $C(\cU,\cV)$, i.e., 
		$\de(c'\otimes c'')=(\de c')\otimes c''+(-1)^q c'\otimes(\de c'')$, 
		for all $c'\in C^q(\cU,\cV)$ and $c''\in C(\cU,\cV)$; 
	\item 
		if the residual $G(\ka)=0$ then the composition $\de\com\de=0$,
		and the differential complex $\{C^q(\cU,\cV),\de^q \mid q\in\bbZ\}$ 
		is defined, with the cohomology spaces 
		$H^q(\cU,\cV)=\ke\de^q\big/\im\de^{q-1}$,  
		where $G(\ka)\in C^2(\cU,\cV)$, 
		\begin{equation*}
			G(k)(\eta',\eta'')=\ka\eta'\com\ka\eta''-\ka(\eta'\com\eta'') 
			\quad\text{for all}\quad \eta',\eta''\in\fM(\cU). 
		\end{equation*}
\end{itemize}
\end{theorem}
\begin{proof} 
The proof is done by the direct verification. 
\end{proof}

\section{The differentiation in algebras and modules.}
\subsection{The differentiation in algebras.}
\begin{defi}
Let $\cA$ be an associative commutative algebra. 
A linear mapping $X\in\End_\bbF(\cA)$ is called a 
{\it  differentiation of the algebra} $\cA$, if the {\it Leibniz rule}
\begin{equation*}
	X(f\cdot g)=Xf\cdot g+f\cdot Xg \quad\text{holds for all\quad} f,g\in\cA. 
\end{equation*} 
Let $\fD(\cA)$ be the set of all differentiations of the algebra $\cA$. 
\end{defi}

The set $\fD(\cA)$ has the natural  structure of a Lie algebra 
with the commuta\-tor, $[X,Y]=X\com Y-Y\com X$, $X,Y\in\fD(\cA)$, as the Lie bracket. 

\begin{prop} 
Let $\cA$ be an associative commutative algebra. 
The following statements hold: 
\begin{itemize} 
	\item 
		the action $\fM(\cA)\to\End_\bbF(\fD(\cA))$ is defined by 
		the composition rule,  
\begin{equation*} 
	\fM(\cA)\ni R\mapsto R\com : \fD(\cA)\to\fD(\cA), \quad X\mapsto R\com X; 
\end{equation*} 
	\item 
		the action $\fD(\cA)\to\fD(\fM(\cA))$ is defined by the commutator rule, 
\begin{equation*} 
	\fD(\cA)\ni X\mapsto [X, \ ] : \fM(\cA)\to\fM(\cA), \quad R\mapsto [X,R]; 
\end{equation*} 
	\item 
		these actions are related by the {\it matching condition} 
\begin{equation*} 
	[X,R\com Y]=[X,R]\com Y+R\com[X,Y] \quad\text{for all}\quad X,Y\in\fD(\cA), \ R\in\fM(\cA). 
\end{equation*}
\end{itemize}
\end{prop}
Thus, the set $\fD(\cA)$ has the structure of a Lie algebra 
and the structure of an $\fM(\cA)$-module, related by the matching condition. 
Briefly, $\fD(\cA)$ is a Lie $\fM(\cA)$-algebra. 

\begin{prop}
Let $\cA$ be an associative commutative algebra. 
The following statements are valid: 
\begin{itemize} 
	\item 
		for any $X\in\fD(\cA)$ the kernel $\ke X=\{f\in\cA\mid Xf=0\}$ is a subalgebra 
		of the algebra $\cA$, while the image $\im X=\{f=Xg\mid g\in\cA\}$ 
		is a $\ke X$-module;
	\item
		for any $X\in\fD(\cA)$ the equality $X(f\cdot g)=f\cdot Xg$ holds 
		for all $f\in\ke X$ and $g\in\cA$,  i.e., $X\in\End_{\ke X}(\cA)$; 
	\item 
		for any $X\in\fD(\cA)$ the image $X(\ann\cA)\subset\ann\cA$.
\end{itemize}
\end{prop}

\subsection{The differentiation in modules.}
\begin{defi}
Let $\cA$ be an associative commutative algebra, and let $\cM$ be an $\cA$-module. 
A pair $\bX=(\na_X,X)\in\E_\bbF(\cM,\cA)=\End_\bbF(\cM)\times\fD(\cA)$ 
is called a {\it differentiation of the $\cA$-module $\cM$}, if the {\it Leibniz rule} 
\begin{equation*}
	\na_X(f\cdot M)=Xf\cdot M+f\cdot\na_X M 
		\quad\text{holds for all}\quad f\in\cA,  \ \ M\in\cM
\end{equation*} 
(remind, $\fD(\cA)\subset\End_\bbF(\cA)$). 
Let $\fD(\cM)$ be the set of all differentiations of the $\cA$-module $\cM$. 
\end{defi} 
The set $\fD(\cM)$  has the structure of a Lie algebra with the component-wise commutator 
$[\bX,\bY]=([\na_X,\na_Y],[X,Y])$ as the Lie bracket. 

\begin{prop} 
Let $\cA$ be an associative commutative algebra, and let $\cM$ be an $\cA$-module. 
The following statements hold: 
\begin{itemize} 
	\item 
		the action $\fM(\cM)\to\End_\bbF(\fD(\cM))$ is defined by the component-wise 
		composition rule, 
\begin{equation*} 
	\fM(\cM)\ni\bR\mapsto\bR :\fD(\cM)\to\fD(\cM), \quad \bX\mapsto\bR\com\bX,
\end{equation*}
		where $\bR=(\De_R,R)$, $\bX=(\na_X,X)$, 
		$\bR\com\bX=(\De_R\com\na_X,R\com X)$; 
	\item 
		the action $\fD(\cM)\to\fD(\fM(\cM))$ is defined by the component-wise 
		commutator rule, 
\begin{equation*} 
	\fD(\cM)\ni\bX\mapsto\bX : \fM(\cM)\to\fM(\cM), \quad \bR\mapsto[\bX,\bR], 
\end{equation*}
		where $\bX=(\na_X,X)$, $\bR=(\De_R,R)$, 
		$[\bX,\bR]=([\na_X,\De_R], [X,R])$; 
	\item 
		these actions are related by the matching condition 
\begin{equation*}
	[\bX,\bR\com\bY]=[\bX,\bR]\com\bY+\bR\com[\bX,\bY] 
\end{equation*} 
for all $\bX,\bY\in\fD(\cM)$, $\bR\in\fM(\cM)$. 
\end{itemize}
\end{prop} 
Thus, the set $\fD(\cM)$ has the structure of a Lie algebra 
and the structure of an $\fM(\cM)$-module, related by the matching condition. 
Briefly, $\fD(\cM)$ is a Lie $\fM(\cM)$-algebra. 

\begin{prop} 
	Let $\cA$ be an associative commutative algebra, and let $\cM$ be an $\cA$-module. 
The following statements hold: 
\begin{itemize} 
	\item 
		for any $\bX=(\na_X,X)\in\fD(\cM)$ the kernel $\ke\na_X=\{M\in\cM\mid \na_X M=0\}$ 
		and the image $\im\na_X=\{M=\na_X N\mid N\in\cM\}$ are $\ke X$-modules; 
	\item 
		for any $\bX=(\na_X,X)\in\fD(\cM)$ the equality 
		$\na_X(f\cdot M)=f\cdot\na_X M$ holds 
		for all $f\in\ke X$ and $M\in\cM$, i.e., $\na_X\in\End_{\ke X}(\cM)$; 
	\item 
		for any $\bX=(\na_X,X)\in\fD(\cM)$ the image $\na_X(\ann_\cM\cA)\subset\ann_\cM\cA$, 
		and the image $X(\ann_\cA\cM)\subset\ann_\cA\cM$.
\end{itemize}
\end{prop}

\begin{prop}\label{P13} 
Let $\cM$ be an $\cA$-module. 
There is the natural Lie $\cA$-algebra injection 
\begin{equation*} 
	\I : \gl_\cA(\cM)\to\fD(\cM), \quad \rho\mapsto(\rho,0), 
		\quad(\text{i.e., } \na_0=\rho).
\end{equation*} 
Moreover, the image $\im\I=\{\bX=(\rho,0)\mid \rho\in\End_\cA(\cM)\}$ 
is an ideal of the Lie $\fM(\cM)$-algebra $\fD(\cM)$, 
thus the short exact sequence of Lie algebras 
\begin{diagram}[2em] 
	0&\rTo&\gl_\cA(\cM)&\rTo^\I&\fD(\cM)&\rTo&\fD(\cM)\big/\im\I&
		\rTo&0
\end{diagram} 
is defined.
\end{prop} 

\begin{prop} 
Let $\cA$ be an associative commutative algebra, and let $\cM$ be a free $\cA$-module 
with an $\cA$-basis ${\mathbf b}=\{b_i\in\cM\mid i\in I\}$, $I$ is an index set, 
so $\cM=\{M=M^i\cdot b_i\mid M^i\in\cA\}$. 
Then a Lie algebra injection $\fD(\cA)\to\fD(\cM)$ is defined by the component-wise action, 
\begin{equation*} 
	\fD(\cA)\ni X\mapsto\bX=(\na_X,X)\in\fD(\cM), 
		\quad \na_X(M^i\cdot b_i)=XM^i\cdot b_i.
\end{equation*}
\end{prop}

There is a nice way to treat differentiations of an $\cA$-module $\cM$. 
Namely, let us consider the mapping  
\begin{equation*} 
	\Pi : \fD(\cM)\to\fD(\cA), \quad \bX=(\na_X,X)\mapsto X.  
\end{equation*}
{\it We assume that the mapping $\Pi$ is the surjection}, i.e., 
$\im\Pi=\fD(\cA)$. In this case, the inverse image 
$\Pi^{-1}(X)=\{\bX=(\na_X,X)\in\fD(\cM)\}$ 
is defined for every $X\in\fD(\cA)$. 
Clear, the mapping 
$\Pi\in\Hom_\cA(\fD(\cM);\fD(\cA))\cap\Hom_{\Lie}(\fD(\cM);\fD(\cA))$, 
i.e., briefly, $\Pi\in\Hom_{\cA\text{-}\Lie}(\fD(\cM);\fD(\cA))$. 

\begin{theorem}\label{T2} 
Let $\cA$ be an associative commutative algebra, 
and let $\cM$ be an $\cA$-module. 
For every $X\in\fD(\cA)$ the inverse image $\Pi^{-1}(X)$ 
is an $\cA$-affine space over the $\cA$-module $\End_\cA(\cM)$. 
Hence, the triple $\Pi : \fD(\cM)\to\fD(\cA)$ is an $\cA$-affine fiber bundle 
over the $\cA$-module $\End_\cA(\cM)$.
\end{theorem}
\begin{proof} 
Indeed, let $X\in\fD(\cA)$, and $\bX'=(\na'_X,X),\bX''=(\na''_X,X)\in\Pi^{-1}(X)$, 
then the difference 
$(\na'_X-\na''_X)(f\cdot M)=(X-X)f\cdot M+f\cdot(\na'_X-\na''_X)M=f\cdot(\na'_X-\na''_X)M$ 
for all $f\in\cA$ and $M\in\cM$, i.e., $\na'_X-\na''_X\in\End_\cA(\cM)$. 
On the other side, let $\bX=(\na_X,X)\in\Pi^{-1}(X)$ and $\rho\in\End_\cA(\cM)$, 
then the sum $(\na_X,X)+(\rho,0)=(\na_X+\rho,X)\in\Pi^{-1}(X)$, 
because $\Pi(\na_X+\rho,X)=X$, and 
$(\na_X+\rho)(f\cdot M)=\na_X(f\cdot M)+\rho(f\cdot M)=Xf\cdot M+f\cdot\na_X M+f\cdot\rho M
=Xf\cdot M+f\cdot(\na_X+\rho)M$ for all $f\in\cA$ and $M\in\cM$.
\end{proof}

We denote by 
\begin{equation*} 
	\cS(\Pi)=\big\{\bna\in\Hom(\fD(\cA);\fD(\cM)) \ \big| \ 
		\Pi\com\!\bna=\id_{\fD(\cA)}\big\}
\end{equation*} 
the set of all {\it sections} of the bundle $\Pi$. In more detail, 
$X\mapsto\bna_X=(\na_X,X)$ for all $X\in\fD(\cA)$. 
There are the following relevant subsections: 
\begin{equation*} 
	\cS_*(\Pi)=\cS(\Pi)\cap\Hom_*(\fD(\cA);\fD(\cM)), 
		\quad *=\bbF, \ \cA, \ \Lie, \ \cA\text{-}\Lie.
\end{equation*}
The 
set $\cS_\bbF(\Pi)$ has the structure of a linear space, 
$\cS_\cA(\Pi)$ has the structure of an $\cA$-module, 
$\cS_{\Lie}(\Pi)$ has the structure of a Lie algebra, 
and $\cS_{\cA\text{-}\Lie}(\Pi)$ has the structure of Lie $\cA$-algebra, 
where the algebraic operations are defined point-wise. 
Elements of the $\cA$-module $\cS_\cA(\Pi)$ are usually called 
{\it covariant derivatives}. 

Let $\cA$ be associative commutative algebra, 
and let $\pmb\cA$ be the adjoint $\cA$-module. Then 
\begin{itemize} 
	\item 
		there is the natural Lie algebra injection \\
		$\io_\fD : \fD(\cA)\to\fD(\pmb\cA), \quad X\mapsto\bX=(X,X)$; 
	\item 
		$\fD(\pmb\cA)=\im\io_\fD\oplus\End_\cA(\pmb\cA), 
		\quad  \bX\!=\!(\na_X,X)\!=\!(X,X)+(\na_X-X,0)$.
\end{itemize} 

\section{The gauge transform.} 
Let $\cA$ be an associative commutative algebra, and let $\cM$ be an $\cA$-module. 
Let $\Aut_\cA(\cM)$ be the group of all automorphisms of the $\cA$-module $\cM$,
i.e., the group of all mappings $G\in\End_\cA(\cM)$ having the inverse $G^{-1}\in\End_\cA(\cM)$, 
$G\com G^{-1}=G^{-1}\com G=\id_\cM$. 

\begin{defi} 
Let $\cA$ be an associative commutative algebra, and let $\cM$ be an $\cA$-module. 
Every mapping $G\in\Aut_\cA(\cM)$ defines the {\it gauge transform} of the 
$\cA$-algebra $\fM(\cM)$ by the rule 
\begin{equation*} 
	\bR=(\De_R,R)\mapsto\ad_G\bR=(\ad_G\De_R,R), 
	\quad \ad_G\De_R=G\com\De_R\com G^{-1}.
\end{equation*}
\end{defi} 

\begin{theorem} 
Let $\cA$ be an associative commutative algebra, and let $\cM$ be an $\cA$-module. 
The gauge transform of the $\cA$-algebra $\fM(\cM)$ defines the action 
\begin{equation*} 
	\ad : \Aut_\cA(\cM)\to\End_{\alg\cap\cA}(\fM(\cM)), \quad G\mapsto \ad_G.
\end{equation*}
\end{theorem} 
\begin{proof} 
Indeed, let $G\in\Aut_\cA(\cM)$, then 
\begin{align*} 
	\ad_G&\De_R(f\cdot M)=(G\com\De_R\com G^{-1})(f\cdot M)
	          	 =(G\com\De_R)(f\cdot G^{-1}M) \\ 
	         &=G(Rf\cdot G^{-1}M)=Rf\cdot(G\com G^{-1}M) =Rf\cdot M, 
\end{align*} 
for all $\bR=(\De_R,R)\in\fM(\cM)$, $f\in\cA$, $M\in\cM$. Further,
\begin{align*} 
	\ad_G&(\bR'\com\bR'')
			=(G\com(\De_{R'}\com\De_{R''})\com G^{-1},R'\com R'') \\
		  &=((G\com\De_{R'}\com G^{-1})\com(G\com\De_{R''}
		  	\com G^{-1}),R'\com R'') =\ad_G\bR'\com\ad_G\bR'', 
\end{align*} 
for all $\bR'=(\De_{R'},R'),\bR''=(\De_{R''},R'')\in\fM(\cM)$. 
At last, 
\begin{align*}  
	\ad_G&\De_R(f\cdot M)=(G\com\De_R\com G^{-1})(f\cdot M) 
	           =f\cdot((G\com\De_R\com G^{-1})M \\
	         &=f\cdot\ad_G\De_R M,
\end{align*} 
for all $\bR=(\De_R,R)\in\fM(\cM)$, $f\in\cA$, $M\in\cM$. 
\end{proof} 

\begin{rem} 
The difference $\ad_G\De_R-\De_R\in\Hom_\cA(\cM;\ann_\cM\cA)$ 
for all  $G\in\Aut_\cA(\cM)$, $\bR=(\De_R,R)\in\fM(\cM)$. 
Indeed, in this case, $\pi(\ad_G\bR)=\pi(\bR)$, see Theorem \ref{T1}. 
\end{rem}

\begin{rem} 
The gauge transform of the $\cA$-algebra $\fM(\cM)$ defines an equivalence relation in $\fM(\cM)$. 
Namely, two multiplicators $\bR'=(\De_{R'},R'),\bR''=(\De_{R''},R'')\in\fM(\cM)$ are called
{\it equivalent}, $\bR'\sim_\fM\bR''$, if $\bR'=\ad_G\bR''$ for some $G\in\Aut_\cA(\cM)$. 
One can check that this relation is really an equivalence relation. 
\end{rem}

\begin{defi} 
Let $\cA$ be an associative commutative algebra, and let $\cM$ be an $\cA$-module. 
Every mapping $G\in\Aut_\cA(\cM)$ defines the {\it gauge transform} of the Lie $\cA$-algebra 
$\fD(\cM)$ by the rule 
\begin{equation*} 
	\bX=(\na_X,X)\mapsto\ad_G\bX=(\ad_G\na_X,X), 
	\quad \ad_G\na_X=G\com\na_X\com G^{-1}. 
\end{equation*}
\end{defi} 

\begin{theorem} 
Let $\cA$ be an associative commutative algebra, and let $\cM$ be an $\cA$-module. 
The gauge transform of the Lie $\cA$-algebra $\fD(\cM)$ defines the action 
\begin{equation*} 
	\ad : \Aut_\cA(\cM)\to\End_{\Lie\cap\cA}(\fD(\cM)), \quad G\mapsto \ad_G.
\end{equation*}
\end{theorem}
\begin{proof} 
Indeed, 
\begin{align*}
	&(\ad_G\na_X)(f\cdot M)=(G\com\na_X\com G^{-1})(f\cdot M) 
		=(G\com\na_X)(f\cdot G^{-1}M) \\
	&=G(\na_X(f\cdot G^{-1}M))
		=G(Xf\cdot G^{-1}M+f\cdot(\na_X\com G^{-1}M)) \\
	&=Xf\cdot(G\com G^{-1}M)+f\cdot(G\com\na_X\com G^{-1})M
		=Xf\cdot M+f\cdot(\ad_G\na_X)M, 
\end{align*}
hence, $\ad_G\bX\in\fD(\cM)$ for all $G\in\Aut_\cA(\cM)$  and $\bX=(\na_X,X)\in\fD(\cM)$. 
Further, 
\begin{equation*} 
	\ad_G(f\na_X)M=(G\com f\na_X\com G^{-1})M
		=f\cdot(G\com\na_X\com G^{-1})M=f\cdot\ad_GM
\end{equation*}
 for all $f\in\cA$, $\bX=(\na_X,X)\in\fD(\cM)$, where $f\bX=(f\na_X,fX)$, 
 hence $\ad_G\in\End_\cA(\fD(\cM))$ for all $G\in\Aut_\cA(\cM)$. 
 At last, let $G\in\Aut_\cA(\cM)$, $\bX'=(\na_{X'},X'),\bX''=(\na_{X''},X'')\in\fD(\cM)$, then 
 \begin{align*} 
 	&[\ad_G\bX',\ad_G\bX'']=([\ad_G\na_{X'},\ad_G\na_{X''}],[X',X'']) \\
 	&=([G\com\na_{X'}\com G^{-1},G\com\na_{X''}\com G^{-1}],[X',X'']) \\
 	&=(G\com[\na_{X'},\na_{X''}]\com G^{-1},[X',X''])=\ad_G[\bX',\bX''], 
 \end{align*}
 i.e., $\ad_G\in\End_{\Lie}(\fD(\cM))$.
\end{proof}

\begin{rem} 
The difference $\ad_G\na_X-\na_X\in\End_\cA(\cM)$ 
for all  $G\in\Aut_\cA(\cM)$, $\bX=(\na_X,X)\in\fD(\cM)$. 
Indeed, in this case, $\Pi(\ad_G\bX)=\Pi(\bX)$, see Theorem \ref{T2}. 
\end{rem}

\begin{rem} 
The gauge transform of the $\cA$-algebra $\fD(\cM)$ defines an equivalence relation in $\fD(\cM)$. 
Namely, two differentiations $\bX'=(\na_{X'},X'),\bX''=(\na_{X''},X'')\in\fD(\cM)$ are called
{\it equivalent}, $\bX'\sim_\fD\bX''$, if $\bX'=\ad_G\bX''$ for some $G\in\Aut_\cA(\cM)$. 
One can check that this relation is really an equivalence relation. 
\end{rem} 

\begin{prop} 
Let $G\in\Aut_\cA(\cM)$, $\bR\in\fM(\cM)$, $\bX\in\fD(\cM)$. Then 
\begin{equation*}
	\ad_G(\bR\com\bX)=\ad_G\bR\com\ad_G\bX \quad\text{and}\quad 
	\ad_G[\bX,\bR]=[\ad_G\bX,\ad_G\bR].
\end{equation*}
\end{prop} 

\section{De Rham cohomology.} 
Let $\cA$ be an associative commutative algebra, let $\cK$, $\cM$ be $\cA$-modules, 
and let $\cU=\cA,\cK$, $\cV=\cA,\cM$. 

\begin{defi} 
The $\cA$-module $\Om(\cU,\cV)=\oplus_{q\in\bbZ}\Om^q(\cU,\cV)$ of {\it differential forms 
over the $\cU$ with coefficients in the $\cV$} is defined by the rule 
\begin{equation*} 
	\Om^q(\cU,\cV)=\begin{cases}
		0, & q<0, \\
		\cV, & q=0, \\
		\Hom_\cA(\w^q_\cA\fD(\cU);\cV), & q>0.
					  \end{cases}
\end{equation*} 
\end{defi} 
In particular, the set $\Om(\cU,\cA)$ has the natural structure 
of an exterior $\cA$-algebra, 
while the set $\Om(\cU,\cM)$ has the natural structure 
of an exterior $\Om(\cU,\cA)$-module. 

\begin{defi}
For every $\xi\in\fD(\cU)$ the {\it interior product} 
$i_\xi\in\End_\cA(\Om(\cU,\cV))$ is defined by the contraction rule 
\begin{equation*} 
	(i_\xi\om)(\xi_1,\dots,\xi_{q-1})=q\om(\xi,\xi_1,\dots,\xi_{q-1})
\end{equation*}
for all $q\in\bbZ$, $\om\in\Om^q(\cU,\cV)$, $\xi_1,\dots,\xi_{q-1}\in\fD(\cU)$. 
\end{defi} 

\begin{prop} 
The following statements hold: 
\begin{itemize} 
	\item 
		$i_{\xi'}\com i_{\xi''}+i_{\xi''}\com i_{\xi'}=0$ 
		for all $\xi',\xi''\in\fD(\cU)$; 
	\item 
		$i_\xi(\phi\w\om)=(i_\xi\phi)\w\om+(-1)^q\phi\w(i_\xi\om)$ 
		for all $\xi\in\fD(\cU)$, $\phi\in\Om^q(\cU,\cA)$ and 
		$\om\in\Om(\cU,\cV)$, 
		for all $\phi\in\Om^q(\cU,\cA)$, i.e., 
		the mapping $i_\xi$ is an exterior differentiation 
		of the exterior algebra $\Om(\cU,\cA)$ 
		and the exterior $\Om(\cU,\cA)$-module $\Om(\cU,\cM)$. 
\end{itemize}
\end{prop} 

\begin{defi}
Let a mapping $\vk\in\Om^1(\cU,\fD(\cV))=\Hom_\cA(\fD(\cU);\fD(\cV))$ 
be fixed. For every $\xi\in\fD(\cU)$ the {\it Lie derivative} 
$L_\xi\in\End_\bbF(\cU,\cV)$ is defined by the rule 
\begin{equation*} 
	(L_\xi\om)(\xi_1,\dots,\xi_q)=(\vk\xi)(\om(\xi_1,\dots,\xi_q))
	-\sum_{1\le r\le q}\om(\xi_1,\dots[\xi,\xi_r]\dots,\xi_q)
\end{equation*} 
for all $q\in\bbZ$, $\om\in\Om^q(\cU,\cV)$, $\xi_1,\dots,\xi_q\in\fD(\cU)$. 
\end{defi}

\begin{prop} 
The following statements hold: 
\begin{itemize} 
	\item 
		$L_\xi(\phi\w\om)=(L_\xi\phi)\w\om+\phi\w(L_\xi\om)$ 
		for all $\phi\in\Om^q(\cU,\cA)$, 
		$q\in\bbZ$, and $\om\in\Om(\cU,\cV)$, i.e., 
		the mapping $L_\xi$ is a differentiation 
		of the exterior algebra $\Om(\cU,\cA)$ 
		and the exterior $\Om(\cU,\cA)$-module $\Om(\cU,\cM)$; 
	\item 
		$[L_{\xi'},i_{\xi''}]=i_{[\xi',\xi'']}$  for all $\xi',\xi''\in\fD(\cU)$; 
	\item 
		if the residual $F(\vk)=0$ then 
		$[L_{\xi'},L_{\xi''}]=L_{[\xi',\xi'']}$ for all $\xi',\xi''\in\fD(\cU)$, 
		where 
		$F(\vk)\in \Om^2(\cU,\fD(\cV))=\Hom_\cA(\w^2_\cA\fD(\cU);\fD(\cV))$, 
		\begin{equation*} 
			F(\vk)(\xi',\xi'')=[\vk\xi',\vk\xi'']-\vk[\xi',\xi''] 
			\quad\text{for all}\quad \xi',\xi''\in\fD(\cU).
		 \end{equation*} 
\end{itemize}
\end{prop} 

 \begin{defi} 
The endomorphism $d=d_\vk\in\End_\bbF(\Om(\cU,\cV))$
is defined by the {\rm Cartan formula}, 
\begin{align*}  
	d\om(\xi_0,\dots,\xi_q)
		&=\frac1{q+1}\bigg\{\sum_{0\le r\le q}(-1)^r
		\vk\xi_r\big(\om(\xi_0,\dots\ck{\xi}_r\dots,\xi_q)\big) \\
		&+\sum_{0\le r<s\le q}\!\!(-1)^{r+s}
		\om([\xi_r,\xi_s],\xi_0,\dots\ck{\xi}_r\dots\ck{\xi}_s\dots,\xi_q)	
		                      \bigg\},
\end{align*}
for all $\om\in\Om^q(\cU,\cV)$, $\xi_0,\dots,\xi_q\in\fD(\cU)$, 
the ``checked'' argument is understood to be omitted, thus  
$d^q=d\big|_{\Om^q(\cU,\cV)} : \Om^q(\cU,\cV)\to\Om^{q+1}(\cU,\cV)$. 
\end{defi}

\begin{theorem}\label{T5} 
The following statements hold: 
\begin{itemize} 
	\item 
		$d(\phi\w\om)=d\phi\w\om+(-1)^q\phi\w d\om$ 
		for all $\phi\in\Om^q(\cU,\cA)$, 
		$q\in\bbZ$, and $\om\in\Om(\cU,\cV)$, i.e., 
		the mapping $d$ is an exterior differentiation 
		of the exterior algebra $\Om(\cU,\cA)$ 
		and the exterior $\Om(\cU,\cA)$-module $\Om(\cU,\cM)$; 
	\item  
		if the residual $F(\vk)=0$ then 
		the composition $d\com d=0$, and 
		the differential complex $\{\Om^q(\cK,\cM),d^q\mid q\in\bbZ\}$ 
		is defined with the cohomology spaces 
		$H^q(\cK,\cM)=\ke d^q\big/\im d^{q-1}$, $q\in\bbZ$. 
\end{itemize}
\end{theorem} 
\begin{proof} 
The proof of the composition $d\com d=0$ based on the assumed equality
$\vk[\xi',\xi'']=[\vk\xi',\vk\xi'']$, $\xi',\xi''\in\fD(\cU)$, of the mapping $\vk$, 
see \cite{Z1} for the detailed exposition. 
The other statements are done by simple verifications. 
\end{proof} 

\begin{theorem} 
Let the residual $F(\vk)=0$ then the {\rm Cartan (magic) formula} 
\begin{equation*} 
	L_\xi=d\com i_\xi+i_\xi\com d
\end{equation*}
holds for any $\xi\in\fD(\cU)$. 
\end{theorem}
\begin{proof} 
The proof is standard, but calculations are rather tiresome. 
See \cite{Z1} for full details. 
\end{proof}
\begin{cor} 
The commutator $[L_\xi,d]=0$ for any $\xi\in\fD(\cU)$. 
\end{cor}

\begin{theorem} 
Let $\cM$ be an $\cA$-module, let $\Om(\cM)=\Om(\cM,\cM)$, 
and let us take $\vk=\id_{\fD(\cM)}\in\End_{\cA}(\fD(\cM))$. 
Then the complex $\{\Om^q(\cM),d^q\mid q\in\bbZ\}$ is exact, 
i.e., the cohomology spaces $H^q(\cM)=H^q(\cM,\cM)=0$ for all $q\in\bbZ$. 
\end{theorem}
\begin{proof} 
Indeed, by Proposition \ref{P13}, the pair $\bE=(\id_\cM,0)\in\fD(\cM)$, 
while $L_\bE=\id_{\Om(\cM)}$. 
Hence, by the Cartan magic formula, the homotopy formula 
\begin{equation*} 
	\id_{\Om(\cM)}=i_\bE\com d+d\com i_\bE
\end{equation*}
holds, implying the claim.
\end{proof}

\end{document}